\documentclass[a4paper,autoref]{lipics-v2021}\nolinenumbers
\usepackage{amsmath,amssymb,amsfonts,latexsym}
\usepackage{color,graphicx}
\usepackage{url} 
\usepackage{fancybox}
\usepackage{algorithm}
\usepackage[noend]{algpseudocode}
\usepackage{mathtools}
\usepackage{enumerate,xspace}
\usepackage{thm-restate}
\usepackage{siunitx} 

\hideLIPIcs

\newcommand{\eps}{\varepsilon}
\renewcommand{\epsilon}{\eps}
\newcommand{\etal}{\emph{et al.}\xspace}

\theoremstyle{plain}

\newenvironment{myquote}%
  {\list{}{\leftmargin=4mm\rightmargin=4mm}\item[]}%
  {\endlist}

\newcommand{\B}{\ensuremath{\mathcal{B}}}

\renewcommand{\leq}{\leqslant}
\renewcommand{\geq}{\geqslant}

\DeclarePairedDelimiter\floor{\lfloor}{\rfloor}

\usepackage[numbers,round]{natbib}

\title{Shift Graphs, Chromatic Number and Acyclic One-Path Orientations}

\author{Arpan Sadhukhan}{Department of Mathematics and Computer Science, TU Eindhoven, the Netherlands}{A.Sadhukhan@tue.nl}{}{}

\authorrunning{A.~Sadhukhan} 
 \Copyright{Arpan Sadhukhan}



\begin{document}

\setcounter{page}{0}
\maketitle

\begin{abstract}
 Shift graphs, which were introduced by Erd\H{o}s and Hajnal~\cite{erdos1968chromatic,ERDOS1966}, have been used to answer various questions in structural graph theory. 
In this paper, we prove two new results using shift graphs and their induced subgraphs. 
\begin{itemize}
\item Recently Girão~\etal~\cite{Girão2023}, showed that for every graph $F$ with at least one edge, there is a constant $c_F$ such that there are graphs of arbitrarily large chromatic number and the same clique number as $F$, in which every $F$-free induced subgraph has chromatic number at most $c_F$.
         We significantly improve the value of the constant $c_F$ for the special case where $F$ is the complete bipartite graph $K_{a,b}$.
         We show that any $K_{a,b}$-free induced subgraph of the triangle-free shift graph $G_{n,2}$ has chromatic
        number bounded by $\mathcal{O}(\log(a+b))$.  
\item An undirected simple graph $G$ is said to have the AOP Property if it can be acyclically oriented such that there is at most one directed path between any two vertices. We prove that the shift graph $G_{n,2}$ does not have the AOP property for any $n\geq 9$. Despite this, we construct induced subgraphs of shift graph $G_{n,2}$ with an arbitrarily high chromatic number and odd-girth that have the AOP property.

\end{itemize}
 Furthermore, we construct graphs with arbitrarily high odd-girth that do not 
 have the AOP Property and also prove the existence of graphs with girth equal to $5$ that do not have the AOP property.
\end{abstract}

\section{Introduction}
The \emph{chromatic number} of a graph $G=(V,E)$ is the minimum number of colors needed to color the vertices of $G$ such that no two neighbors have the same color. The chromatic number of (graphs from) various graph classes has been a central topic of study in graph theory. Typically, the graph classes are defined through forbidden subgraphs, either as a minor or as an induced subgraph. The study of graphs with a high chromatic number has been of fundamental importance, especially if anything can be said about their local structure. There has been a sea of results and conjectures in this field. For the interested reader, open problems, progress, and various alternative constructions are well-documented in a survey by Scott and Seymour~\cite{DBLP:journals/jgt/ScottS20}.  

We say a graph $G$ is $F$-free if it does not have an induced subgraph isomorphic to $F$.
It is interesting to see how forbidding a local structure influences the chromatic number of a graph. Erdős~\cite{erdos_1959} first proved the existence of graphs with arbitrarily high girth and chromatic number. This trivially shows that for any graph $F$ containing a cycle, there are graphs $G$ with an arbitrarily large chromatic number that does not contain $F$ as its (not necessarily induced) subgraph. Now, on the other hand, one may ask whether there is a constant $c_F$ depending only on $F$ such that there exist graphs $G$ with an arbitrarily large chromatic number such that any $F$-free induced subgraph of $G$ has chromatic number at most~$c_F$. The answer to this is positive even with an additional restriction that the clique number of $G$ is the same as the clique number of $F$. In this regard, Carbonero~\etal~\cite{CARBONERO202363} constructed graphs whose clique number is bounded by $3$, have an arbitrarily high chromatic number, and the triangle-free induced subgraph of such graphs have chromatic number bounded by $4$. Girão~\etal~\cite{Girão2023} generalized the result and proved that for every graph $F$ with at least one edge, there is a constant $c_F$ such that there are graphs of arbitrarily large chromatic number and the same clique number as $F$, in which every $F$-free induced subgraph has chromatic number at most $c_F$. The upper bound for $c_F$ as given by Girão~\etal is $O(V(F)^{4E(F)})$. Girão~\etal observed that this bound can be further improved to $O((V(F))^9)$ using results in~\cite{Davies23}, where bounds for $c_F$ are obtained when $F$ is a clique. In this paper, we prove better bounds for $c_F$ when $F$ is a complete bipartite graph. We use the well-known shift graphs $G_{n,2}$ for this. 

Let $n$ and $k$ be integers with $n > 2k > 2$. The shift graph $G_{n,k}$ is a graph whose vertex set is the set of all $k$-tuples $(a_1, a_2, \ldots, a_k)$ of natural numbers such that $1 \leq a_1 < a_2< \cdots < a_k \leq n$. In this graph, $(a_1, a_2, \ldots, a_k)$ and $(b_1, b_2, \ldots, b_k)$ are adjacent if $a_{i+1} = b_i$ for all $1 \leq i < k$, or vice versa. This graph defined by Erd\H{o}s and Hajnal~\cite{erdos1968chromatic,ERDOS1966}, has nice properties. For example, for a fixed $k$, its chromatic number tends to
infinity with $n$. Moreover, all its odd cycles have lengths of at least $2k + 1$. Some applications of shift graphs include lower bounds in Ramsey theory, and constructions of infinite graphs with interesting properties relating to chromatic number. These applications along with generalizations and other properties of shift graphs can be found in~\cite{DBLP:journals/combinatorics/ArmanRS22,  ChristianAvart2014,DBLP:journals/dm/DuffusLR95,furedi1992interval}.

In this paper, we prove that the shift graphs $G_{n,2}$ satisfy the property that any $K_{a,b}$-free induced subgraph of it has chromatic number bounded by  $\mathcal{O}(\log(a+b))$. Since $G_{n,2}$ is also triangle-free, the result implies that if $F=K_{a,b}$, then we have $c_F$ bounded by $\mathcal{O}(\log(a+b))$, which significantly improves the current known bound for complete bipartite graphs. It is also conjectured~\cite{Girão2023} that the constant $c_F$ must be bounded by a function of $\chi(F)$ if the graph $F$ is not vertex critical. A stronger conjecture is also made in the same paper by Nešetřil, which says that the constant $c_F$ must be bounded by $\chi(F)$ if the graph $F$ is not vertex critical. But we argue at the beginning of section~\ref{subsec: K_{a,b} hole Chromatic number} that this stronger form of the conjecture does not hold when $F$ is a complete bipartite graph. 

It is well known that a graph $G$ is $k$-colorable if and only if it has an acyclic orientation $\overrightarrow{G}$ which contains no directed paths of length $k$. This establishes a direct connection between the chromatic number and acyclic orientations of a graph. Moreover, a common theme of all the constructions in~\cite{Davies23, CARBONERO202363, Girão2023} is that they use graphs that can be acyclically oriented such that between any two vertices there exists at most one directed path. We say that a graph $G$ has the \emph{AOP (acyclic one path) property} if it can be acyclically oriented such that there is at most one directed path between any two vertices. It is known that the triangle-free Zykov graphs have the AOP property~\cite{DBLP:journals/dm/KiersteadT92}. We believe it is important to study and explore which graphs have the AOP property. In the section~\ref{subsec:AOP}, we prove that there exist graphs with arbitrary high odd-girth (length of the smallest odd cycle) which do not have the AOP property. We also prove that there exist graphs with girth (length of the smallest cycle) equal to $5$ that do not have the AOP property. Then we study the AOP property in the context of shift graphs. We show that the shift graphs $G_{n,2}$ do not have the AOP property for any $n\geq 9$. Although the shift graphs seem to resist the AOP property at first glance, surprisingly enough we manage to construct induced subgraphs of shift graphs with arbitrary high chromatic number and odd-girth that have the AOP property. We also conjecture that there exist graphs with arbitrarily high girth that do not have the AOP property.

\section{Preliminaries} \label{construction}
We first give a few standard definitions that will be used throughout the paper. Let $G$ be a graph with vertex set $V(G)$ and edge set $E(G)$. In this paper, all graphs are finite and simple.  We say a graph $G$ is \emph{$H$-free} if it does not have an induced subgraph isomorphic to $H$. We denote the \emph{chromatic number} of $G$ by $\chi(G)$ and its \emph{clique number} by $\omega(G)$. We denote the \emph{neighborhood} of a vertex $v\in G$ (that is, the set of all vertices that are adjacent to $v$) by $N_G(v)$. Note that $v\not\in N_G(v)$. For $X \subseteq V(G)$, let $G[X]$ denote the subgraph induced by $X$. All directed graphs in this paper refer to \emph{oriented} graphs. Thus for a pair of nodes $u,v$ there can be at most one
directed edge.
Directed graphs (digraphs for short) will be denoted by $\vec{G}$, 
where $G$ is the undirected counterpart of $\vec{G}$. In other words, $\vec{G}$ is obtained from $G$ by \emph{orienting} each of its edges. The directed edge (or: \emph{arc}) from a node $a$ to node $b$ is written as $ab$.
We say that $\vec{G}$ is \emph{acyclic} iff it has no directed cycles. The \emph{girth} of a graph $G$ is the length of a shortest cycle contained in $G$ (it is zero if the graph contains no cycle). The \emph{odd-girth} 
of a graph $G$ is the number of edges in a shortest odd cycle of $G$. The undirected girth/odd-girth of a directed graph $\vec{G}$ is the girth/odd-girth of the underlying undirected graph $G$. A \emph{subdigraph} of a digraph $\vec{G}$ is a subgraph of the underlying undirected graph $G$, with an orientation induced from the orientation $\vec{G}$ of $G$. 

The \emph{line digraph} $\vec{L}(\vec{G})$ of a digraph $\vec{G}$ is a digraph that has as its vertices the arcs of $\vec{G}$, and has an arc from $ab \in E(\vec{G})$ to
$cd \in E(\vec{G})$, iff $b = c$. Let $L(\vec{G})$ denote the underlying undirected graph of the line digraph $\vec{L}(\vec{G})$ (Note that the graph $L(\vec{G})$ depends on an orientation $\vec{G}$ of the graph $G$ and not just $G$ itself). A \emph{topological ordering} of $\vec{G} = (V(\vec{G}), E(\vec{G}))$ is an ordering of its nodes as $v_1, v_2, \ldots, v_n$, such that if there is an arc between $v_i$ and $v_j$, then that arc is directed from $v_i$ to $v_j$ iff $i < j$. It is well known that a digraph is acyclic iff it has a topological ordering. 

Let $\vec{G}$ be a directed acyclic graph with $n$ vertices. Let $\sigma=\{v_1,v_2, v_3, \ldots, v_n\}$ be a topological ordering of the vertices of $\vec{G}$. Let $B(i)$ be the set of outgoing arcs of the vertex $v_{i}$ for all $1 \leq i \leq n-1$. It is clear that the set $B(i)$ forms an independent set in $L(\vec{G})$. We call  $\B_{\sigma}=\{B(i): i=1, 2, \ldots , n-1\}$ the \emph{bag decomposition} of $L(\vec{G})$ with respect to the chosen topological ordering $\sigma$ and we refer to the sets $B(i)$ as the \emph{bags} of $L(\vec{G})$. See Figure~\ref{fig:shift of graph} for an illustration.

Now we discuss some basic structural properties of $L(\vec{G})$, when $\vec{G}$ is acyclic.
 It is known that $G_{n,2}=L(\vec{T}_n)$~\cite{DBLP:books/daglib/0013017}, where $\vec{T}_n$ is the acyclic tournament on $n$ vertices. Also, $\vec{G}_{n,2}=\vec{L}(\vec{T_n})$, the orientation of $G_{n,2}$, is referred to as the natural orientation of the shift graph.

\begin{observation}\label{the shift graph}
    The graph $L(\vec{T}_n)$ is isomorphic to the shift graph $G_{n,2}$. Also, the generalized shift graph $G_{n,k}$ is isomorphic to $L^{k-1}(\vec{T}_n)$
\end{observation}

\begin{observation} \label{the induced shift graph}
    For any directed acyclic graph $\vec{G}$ on $n$ vertices, the graph $L(\vec{G})$ is an induced subgraph of the shift graph $L(\vec{T}_n)$. 
\end{observation}
\begin{proof}
    It trivially follows from the fact that we can always delete edges of $\vec{T}_n$ to get the graph $\vec{G}$.
\end{proof}

It is easy to see that the directed acyclic complete graph $\vec{T}_n$ has a unique topological ordering of its vertices. With respect to that ordering, $L(\vec{T}_n)$ has $n-1$ non-empty bags, and for the $i$-th bag $B(i)$, we have $|B(i)|=n-i$. Furthermore, for every $i<j$, each bag $B(i)$ has a unique vertex that is adjacent to all the vertices of the bag $B(j)$. Now for a directed acyclic graph $\vec{G}$, given the bag decomposition of $L(\vec{G})$, we also define an index function $I$ on the vertices of $L(\vec{G})$ that indicates which bag of $L(\vec{G})$ the vertex belongs to. Thus, for any vertex $u\in L(\vec{G})$ we have $I(u)=i$ if $u \in B(i)$. Note that the function is well-defined because of the definition of $B(i)$.  
Some basic structural properties of $L(\vec{G})$ and $\vec{L}(\vec{G})$, where $\vec{G}$ is a directed acyclic graph are noted in the following observation and will be used throughout the paper. 

\begin{observation} \label{structure of S(G)}
 Let $\sigma=\{v_1,v_2, \ldots, v_n\}$ be a topological ordering of the vertices of a directed acyclic graph $\vec{G}$. Let $\B_{\sigma}  =\{B(i): i=1, 2, \ldots , n-1\}$ be the bag decomposition of $L(\vec{G})$. Then the following holds:

\begin{enumerate}[(i)]

    \item  For any $1\leq i \leq n$, all the vertices in the bag $B(i)$ form an independent set of $L(\vec{G})$.
    
    \item  If $u_1, u_2 \in L(\vec{G})$ are adjacent, then $v_{I(u_1)}$ and $v_{I(u_2)}$ are adjacent in $G$. Furthermore in $\vec{L}(\vec{G})$ the edge is directed from $u_1$ to $u_2$ iff $I(u_1)< I(u_2)$.
    
\item  Suppose that $u \in L(\vec{G})$ is adjacent to two vertices $u_1, u_2 \in L(\vec{G})$. If $I(u_1) \geq I(u) \; \text{and} \; I(u_2) \geq I(u)$ then $I(u_1)=I(u_2)$.

\item If $i<j$ and $v_i$ is adjacent to $v_j$ in $G$ then there exists exactly one vertex in $B(i)$ that is adjacent to all vertices of $B(j)$ and no other vertex of $B(i)$ is adjacent to any vertex of $B(j)$.

\item Let $u_1, u_2, u_3 \in L(\vec{G})$ such that $u_1$ is adjacent to $u_2, u_3$. Suppose $I(u_1)>I(u_2)$ and $I(u_1)>I(u_3)$. Then $u_2$ and $u_3$ are adjacent to all the vertices of $B(I(u_1))$ and $u_2$ and $u_3$ are non-adjacent vertices of $L(\vec{G})$.

\end{enumerate}
\end{observation}
\begin{proof}

\begin{enumerate}[(i)]
    \item The set $B(i)$ contains all the outgoing arcs from the vertex $v_{i}$ in $\vec{G}$. So clearly they cannot be adjacent, as the end point of one arc is not the start of the other arc. 
    
    \item This directly follows from the definition of the line graph and line digraph.
     
    \item Any vertex of $B(I(u))$ corresponds to an outgoing arc of $v_{I(u)}$ in $\vec{G}$ and this arc can only share its endpoint with the starting point of the outgoing arcs from a single vertex. Hence, it cannot be adjacent to the outgoing arcs of two distinct vertices $v_i$ and $v_j$ if $i,j > I(u)$. 
      
    \item $B(i)$ contains the outgoing arcs of the vertex $v_i$ in $\vec{G}$ and since $G$ is a simple graph, only one arc shares its endpoint with all the outgoing arcs of $v_j$ in $\vec{G}$.
       
   \item Without loss of generality assume $I(u_2)>I(u_3)$ and suppose for the sake of contradiction $u_2$ and $u_3$ are adjacent. Since we also have $I(u_1)>I(u_3)$, by (iii) we know that $I(u_1)=I(u_2)$ which is a contradiction. The fact that $u_2$ and $u_3$ are adjacent to all the vertices of $B(I(u_1))$ follows from (iv). \end{enumerate}
This finishes the proof of all the properties.
\end{proof}

\begin{figure}
\begin{center}
\includegraphics{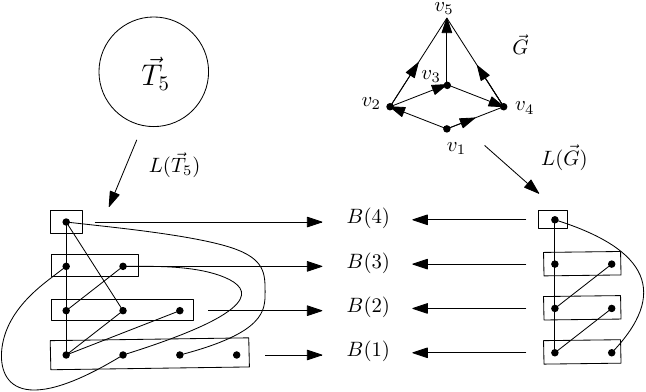}
\end{center}
\caption{The graph $L(\vec{T}_5)$ is shown in the left and in the right the graph $L(\vec{G})$ is shown for the directed acyclic graph $\vec{G}$(top right), with $\{v_1, v_2, v_3, v_4, v_5\}$ representing a topological ordering of the vertices of $\vec{G}$}.

\label{fig:shift of graph}
\end{figure}

We now note another useful observation.
\begin{observation}\label{L(G) acyclic}
   Let $\vec{G}$ be a directed acyclic graph. Then $\vec{L}(\vec{G})$ is also acyclic.
\end{observation}
\begin{proof}
    Let $\sigma'=\{u_1, u_2, \ldots, u_m\}$ be an ordering of the vertices of $\vec{L}(\vec{G})$ such that for all $i<j$, $I(u_i) \leq I(u_j)$. By Observation~\ref{structure of S(G)}(ii), the ordering $\sigma'$ is a topological ordering of the vertices of $\vec{L}(\vec{G})$. Hence $\vec{L}(\vec{G})$ is acyclic.
\end{proof}
 
Given a directed acyclic graph $\vec{G}$, Observation~\ref{L(G) acyclic} allows us to define \emph{iterated line digraphs} as follows. We define $\vec{L}^g(\vec{G})=\vec{L}(\vec{L}^{g-1}(\vec{G}))$, with $\vec{L}^1(\vec{G})=L(\vec{G})$ and $\vec{L}^0(\vec{G})=\vec{G}$. Also let $L^g(\vec{G})$ denote the underlying undirected graph of $\vec{L}^g(\vec{G})$.


\section{Some standard properties of $L(\vec{G})$} \label{subsec:Chromatic number}

In this section, we discuss some fundamental properties of the graph $L(\vec{G})$, which are inherited from the directed acyclic graph $\vec{G}$. Most results in this section are either present in literature~\cite{DBLP:books/daglib/0013017} or are very similar to the corresponding theorems about shift graphs, but we still prove them for the sake of completeness.

The lemma below bounds $\chi(L(\vec{G}))$ in terms of $\chi(G)$, where $\vec{L}(\vec{G})$ is the line digraph of $\vec{G}$ ($\vec{G}$ need not be acyclic here). The following lemma is proved in~\cite{DBLP:books/daglib/0013017}.

\begin{lemma}[\cite{DBLP:books/daglib/0013017}] {\label{chromatic number of S(G)}}
   Let $k^*=\min_{k \in \mathbb{N}}: \chi(G) \leq {k \choose \floor{\frac{k}{2}}}$. Then $ \log_{2}(\chi(G)) \leq \chi(L(\vec{G})) \leq k^*=\mathcal{O}(\log(\chi(G)))$.
\end{lemma}

\begin{corollary} \label{chromatic iteration}
    Let $\vec{G}$ be an acyclic directed graph. Then for any $g\in \mathbb{N}$, we have that $\chi(L^g(\vec{G})) \geq \log_{2}\log
    _{2}\cdots\log_{2} (\chi(G))$, where the logarithm is taken $g$ times.
\end{corollary}

\begin{proof}
  Using the lower bound in Lemma \ref{chromatic number of S(G)}  iteratively $g$ times we get our result.
\end{proof}

It is well known that the odd-girth of the generalized shift graph $G_{n,k}=L^{k-1}(\vec{T}_n)$ is at least $2k-1$, where $\vec{T}_n$ is the complete acyclic graph. It is actually because of a more general result that for any directed acyclic graph $\vec{G}$, the odd-girth of $L(\vec{G})$ is strictly bigger than the undirected odd-girth of $\vec{G}$, which we prove now. We give some standard definitions that are needed for our proof.

A \emph{walk} $W$ on a graph $G$ is a sequence $v_1,e_1,v_2,e_2,...,e_{m-1},v_m$ that alternates between vertices and edges of $G$, beginning and ending with a vertex and each edge $e_j$ of $W$ is incident with the vertex $v_j$ and the vertex $v_{j+1}$. In simple graphs only specifying the vertices is enough as the edges between two vertices are unique. A closed walk is a walk that starts and ends in the same vertex. Also, the length of a walk is the number of edges in it.
 Before we proceed to our theorem, we state a standard result below. For the sake of completeness, we give the proof.

\begin{observation} \label{odd cycle}
  In any simple undirected graph $G$, every closed walk with an odd length contains an odd cycle.
\end{observation}
\begin{proof}
    We proceed by induction on the length of the walk. Trivially, in a simple graph, there cannot be a closed odd walk of length $1$. So we have our base case when the closed walk $v_1, e_1, v_2, e_2, v_3, e_3, v_1$ has length $3$. Since $G$ is simple, $v_2 \neq v_1$ and $v_3\neq v_2 $ and $v_3\neq v_1$. So the walk is an odd cycle of length $3$.

Now, suppose any closed odd walk of length $2g-1$ has an odd cycle. Let $$W=v_1, e_1, v_2, e_2, v_3, e_3, \ldots v_{2g+1}, e_{2g+1}, v_1$$ be a closed walk of length $2g+1$. If all the vertices $v_i$, $1 \leq i \leq 2g+1$ are distinct then $W$ is an odd cycle and we are done. Now suppose that $W$ has two repeated vertex (say $v_i=v_j$, $i<j$). Then clearly they cannot be consecutive as $G$ is simple. So we can divide the walk $W$ into two distinct closed walks $W_1$, $W_2$ both starting with $v_i$ and ending in $v_j$ where $W_1=v_i, v_{i+1}, \ldots, v_j$ and $W_2=v_i, v_{i-1}, \ldots , v_1, v_{2g+1}, \ldots, v_j$. Since each of the walks has a size greater than or equal to $2$, each walk will have a size less than or equal to $2g-1$, and trivially one of them must be an odd walk. Hence, by induction, we are done.
\end{proof}

\begin{lemma} \label{odd girth}
    If a directed acyclic graph $\vec{G}$ has an undirected odd-girth equal to $2g-1$, then the odd-girth of $L(\vec{G})$ is greater than or equal to $2g+1$.
\end{lemma}
\begin{proof}
     Let  $\sigma=\{v_1,v_2, \ldots, v_n\}$ be a topological ordering of $\vec{G}$. Let $\B_{\sigma}  =\{B(i): i=1, 2, \ldots , n-1\}$ be the bag decomposition of $L(\vec{G})$.

Suppose for a contradiction, $C=\{u_1, u_2, \ldots, u_{m}\}$ be an odd cycle of $L(\vec{G})$ with $m \leq 2g-1$. 
By Observation~\ref{structure of S(G)}(ii), we know that $W= v_{I(u_1)}, v_{I(u_2)}, v_{I(u_3)}, \ldots , v_{I(u_m)}, v_{I(u_1)}$ forms a closed walk in $G$. We will show that the indices $I(u_{j})$ for $j=1, 2, \ldots, m$ cannot be all distinct.
 Let $I(u_{j^*})= \min_{1 \leq j \leq m}I(u_j)$. Let $u_{j^*_1}$ and $u_{j^*_2}$ be two neighbour of $u_{j^*}$ in the cycle $C$. Since $I(u_{j^*})$ is minimum among the index values of the vertices in the cycle $C$, by Observation~\ref{structure of S(G)}(iii), we have $I(u_{j^*_1})=I(u_{j^*_2})$. Hence $W$ contains repeated vertices and therefore is not a cycle in $G$. Moreover, by Observation~\ref{odd cycle}, we know that $G$ has an odd cycle consisting of vertices from $W$, so the length of that odd cycle must be strictly less than $2g-1$ which is a contradiction because the odd girth of $G$ is $2g-1$. Hence, any odd cycle of $L(\vec{G})$ must have a length greater than or equal to $2g+1$.  \end{proof}

\begin{corollary}\label{girth iteration}
    For any directed acyclic graph $\vec{G}$, the odd-girth of the graph $L^g(\vec{G})$ is greater than or equal to $2g+1$.
\end{corollary}

\begin{proof}
Trivially follows from repeated application of Lemma~\ref{odd girth}.
\end{proof}






\section{Chromatic number of $K_{a,b}$-free induced subgraph of the shift graphs}\label{subsec: K_{a,b} hole Chromatic number}

Girão~\etal~\cite{Girão2023}, proved the following theorem.
\begin{theorem}[\cite{Girão2023}]\label{girao}
     For every graph $F$ with at least one edge, there is a constant $c_F$ such that there are graphs of arbitrarily large chromatic number and the same clique number as $F$, in which every $F$-free induced subgraph has chromatic number at most $c_F$.
\end{theorem}

In this section, we consider the case when $F$ is a complete bipartite graph. Note that the odd cycle of the minimum length in a triangle-free graph with chromatic number greater than or equal to $3$ must trivially be an induced odd cycle. Now any odd cycle has chromatic number $3$ and does not have an induced $K_{a,b}$ for $a,b \geq 2$. Hence, the constant $c_F$ in Theorem~\ref{girao} is trivially lower bounded by $3$ for $F=K_{a,b}$ for $a,b \geq 2$. This disproves a conjecture of Nešetřil in~\cite{Girão2023} which states that the constant $c_F$ is equal to $\chi(F)$, when $F$ is not vertex critical. The current known upper bound for the constant $c_F$ in Theorem~\ref{girao} for any graph $F$ is $O(V(F)^9)$. Here we improve this bound for $F=K_{a,b}$ and show that the constant $c_F$ in Theorem~\ref{girao} is upper bounded by $O(\log(V(F)))$ in this case.
More specifically, we prove that for any $n\in \mathbb{N}$, the chromatic number of any induced subgraph of the triangle-free shift graph $G_{n,2}$ which does not contain an induced $K_{a,b}$ is $O(\log(a+b))$. 

Before we move on to the proof, we would like to note there are multiple ways (excluding the more general result proved in~\cite{Girão2023}) of proving the existence of triangle-free graphs with large chromatic number in which any $K_{a,b}$-free induced subgraph has bounded chromatic number. But the bounds obtained from those proofs are much worse that ours. We sketch two of them below\footnote{The first proof was communicated by the anonymous reviewer and the second proof was communicated by Sophie Spirkl. We are grateful to both of them for pointing them out.} for completeness, which although somewhat known, was not documented specifically anywhere to the best of our knowledge.

 First observe that if a triangle-free graph forbids a $K_{a,b}$ as a induced subgraph then it trivially forbids $K_{a,b}$ as a subgraph. Also it is easy to see that every $k$-chromatic graph has a $k$-chromatic subgraph of minimum degree at least $k-1$. Now for the first proof, we first note that for any $k\in \mathbb{N}$, Pawlik~\etal~\cite{PAWLIK20146} constructed triangle-free graphs $G_k$ with chromatic number greater than or equal to $k$, that can also be represented as line segment intersection graphs in the plane. Also in~\cite{PAWLIK20146} it is noted that if we take $H$ be a 1-subdivision of some non-planar graph then no subdivision of such $H$ is representable as an intersection graph of segments. Thus by the observations above, it is easy to see that a result of K{\"{u}}hn and Osthus (Theorem $1$ in~\cite{DBLP:journals/combinatorica/KuhnO04a}) directly proves that the $K_{a,b}$-free induced subgraph of the triangle-free graphs $G_k$ constructed in~\cite{PAWLIK20146} must have bounded chromatic number. 

For the second proof, first note that Gyárfás~\cite{Gyárfás115} observed that the triangle-free shift graph $G_{n,2}$ with its natural orientation forbids a $4$ vertex path $\vec{P}_4$ oriented $\rightarrow\leftarrow\rightarrow$ as an induced subgraph. By a result of Burr~\cite{Burr},  we know that any orientation of a $k^2$-chromatic graph contains each oriented $k$-edge tree. Now consider the oriented tree $\vec{T}$ with $a+b+1$ edges, which has two vertices $p,p'$ with an edge between them oriented from $p'$ to $p$ and which has a set $A$ consisting of $a$ vertices (other than $p'$) which are in-neighbours of $p$ and a set $B$ consisting of $b$ vertices (other than $p$) which are out-neighbours of $p'$. Now let $H$ be a $(a+b+1)^2$-chromatic induced subgraph of $G_{n,2}$, and orient the edges of $H$ using the induced natural orientation to get the oriented graph $\vec{H}$. Now by the result of Burr, $\vec{H}$ must contain $\vec{T}$ as a subgraph. But this implies that the sets of vertices $A$ and $B$ of $\vec{T}$ present in $V(H)$ must form a complete bipartite graph $K_{a,b}$ (otherwise, we either get a triangle or a $4$ vertex path $\vec{P}_4$ oriented $\rightarrow\leftarrow\rightarrow$ which is a contradiction by the observation of Gyárfás~\cite{Gyárfás115} as noted above). Thus we get a bound of $(a+b+1)^2$ for $K_{a,b}$-free induced subgraph of $G_{n,2}$ by this argument.

Now we move to the proof of the main result of this section that give a much better bound that the ones we get from the arguments above. Recall that the \emph{degeneracy} of a graph $G$ is the least $k$ such that every induced subgraph of $G$ contains a vertex with $k$ or fewer neighbors. It is well known that if the degeneracy of a graph $G$ is $k$ then the chromatic number of the graph is bounded by~$k+1$.

\begin{theorem} \label{Complete bipartite free induced}
Any $K_{a,b}$-free induced subgraph $H$ of the shift graph $G_{n,2}$ satisfies $\chi(H)= \mathcal{O}(\log(a+b))$.
\end{theorem}
\begin{proof}
We know from Observation~\ref{the shift graph} that $G_{n,2}=L(\vec{T}_n)$, where $\vec{T}_n$ is the directed acyclic complete graph on $n$ vertices. Also, let $\vec{T}'$ be the subdigraph of $\vec{T}_n$ such that $H=L(\vec{T}')$.
    Let  $v_1,v_2, \ldots, v_n$ be the topological ordering $\sigma$ of the vertices of the directed acyclic complete graph $\vec{T_n}$. Let $\B_{\sigma}  =\{B(i): i=1, 2, \ldots , n-1\}$ be the bag decomposition of $L(\vec{T}_n)$.
Let $B_{H}(v_i)=B(i) \cap V(H), \; \text{for all} \; i=1, 2, \ldots , n-1$. Now we partition $V(\vec{T}_n)$ into two sets $L$ and $R$ as follows. 
Let $L:=\{ v_i \subseteq V(\vec{T}_n):  |B_{H}(v_i)|\leq b-1\}$ and $R:=\{ v_i \subseteq V(\vec{T}_n): |B_{H}(v_i)| \geq b \}$. Thus $L$ denotes the set of all vertices of $\vec{T}_n$ such that $B_{H}(v_i)$ has size at most $b-1$ and $R$ denotes the set of all vertices of $\vec{T}_n$ such that $B_{H}(v_i)$ has at least $b$ vertices.

Let $T'[L]$ denote the induced subgraph of $T'$ on $L$. Similarly let $T'[R]$ denote the induced subgraph of $T'$ on $R$. \\

\emph{Claim 1:} $\chi(T'[L])\leq b$.

\begin{claimproof}
      Let $\sigma_1=\{v_{i_1}, v_{i_2}, \ldots, v_{i_m}\}$ be an ordering of all the vertices of $T'[L]$ such that $1 \leq i_k < i_{k'}\leq n$, for all $k>k'$.
 For any $1\leq k \leq m$, let $S(v_{i_k})$ denote the subset of vertices in $\{v_{i_1}, v_{i_2}, \ldots, v_{i_{k-1}}\}$ to which $v_{i_k}$ is adjacent in $T'[L]$. So, $S(v_{i_1})$ is empty. Now Observation~\ref{structure of S(G)}(iii) and $|B_{H}(v_{i_k})|\leq b-1$ together implies $|S(v_{i_k})|\leq b-1$ for all $k=1, 2, 3, \ldots, m$. Given any induced subgraph $H'$ of $T'[L]$, let $m^* \leq m$ be such that $v_{i_{m^*}} \in H'$ and $v_{i_j} \notin H'$ for any $j> m^*$. Then $\mathrm{degree}(v_{i_{m^*}})\leq b-1$. Hence the degeneracy of $T'[L]$ is bounded by $b-1$. Thus, $\chi{(T'[L])}\leq b$.
\end{claimproof}

\emph{Claim 2: $\chi{(T'[R])}\leq a$.}

\begin{claimproof}
    Let $\sigma_1=\{v_{i_1}, v_{i_2}, \ldots, v_{i_{m'}}\}$ be an ordering of all the vertices of $T'[R]$ such that $1 \leq i_k < i_{k'}\leq n$ for all $k<k'$. For any $1\leq k \leq m'$, let $S(v_{i_k})$ denote the subset of vertices in $\{v_{i_1}, v_{i_2}, \ldots, v_{i_{k-1}}\}$ to which $v_{i_k}$ is adjacent in $T'[R]$. So, $S(v_{i_1})$ is empty.

    Now observe that if $v_{i_z}$ is adjacent to $v_{i_{z'}}$ in $T'[R]$, with $i_{z'}<i_{z}$ then by Observation~\ref{structure of S(G)}(v), we know that exactly one vertex of $B_{H}(v_{i_{z'}})$ is adjacent to all vertices of $B_{H}(v_{i_z}))$.
    Hence if $|S(v_{i_k})| \geq a$, then there exists at least $a$ distinct vertices of $H$ such that each of them is adjacent to all the vertices of  $B_{H}(v_{i_k})$. Now since  $|B_{H}(v_{i_k})| \geq b$, we get a $K_{a,b}$ as a subgraph in $H$ (which must be induced as the graph is triangle-free), which gives a contradiction. So we have $|S(v_{i_k})| \leq a-1$ for all $k=1, 2, 3, \ldots, m'$. Now given any induced subgraph $H'$ of $T'[R]$, let $m^*\leq m'$ be such that $v_{i_{m^*}} \in H'$ and $v_{i_j} \notin H'$ for any $j> m^*$. Then $\mathrm{degree}(v_{i_{m^*}})\leq a-1$. Hence the degeneracy of $T'[R]$ is bounded by $a-1$. Thus, $\chi{(T'[R])}\leq a$.
\end{claimproof}

Thus $\chi(T')\leq a+b$. Now since  $H=L(\vec{T}')$, Lemma~\ref{chromatic number of S(G)} implies the theorem.
\end{proof}

Since $G_{n,2}$ is triangle-free by Lemma~\ref{odd girth}, the clique number of $G_{n,2}$ is the same as the clique number of $K_{a,b}$. Moreover, by Lemma~\ref{chromatic number of S(G)}, we know that the chromatic number of $G_{n,2}$ tends to infinity as $n$ tends to infinity. Hence, the conditions of Theorem~\ref{girao} are satisfied and Theorem~\ref{Complete bipartite free induced} implies that the constant $c_F$ in Theorem~\ref{girao} is bounded by $\mathcal{O}(\log(a+b))$ if $F=K_{a,b}$.

\section{AOP property of graphs} \label{subsec:AOP}
In this section, we study special acyclic orientations of simple undirected graphs.
It is known that there is a direct connection between chromatic number and acyclic orientations of a graph. This connection is given by the following proposition which is the acyclic case of the Gallai-Roy-Hasse-Vitaver theorem~\cite{DBLP:books/daglib/0030491}.

\begin{proposition}[\cite{Girão2023}]\label{coloring and orientation}
    A graph L is $k$-colorable if and only if it has an acyclic orientation $\vec{L}$ that contains no directed paths of length $k$.
\end{proposition}

Recall that a simple undirected graph is said to have the AOP property if it can be oriented acyclically such that there exists at most one directed path between any pair of vertices. It is easy to observe that graphs having AOP property must be triangle-free, as it is impossible to acyclically orient a triangle such that any two vertices have at most one directed path between them. Moreover, if a graph $G$ does not have the AOP property then in any acyclic orientation $\vec{G}$ of $G$, there exists two vertices $u,v \in \vec{G}$ such that there are at least two disjoint directed paths between them. Kierstead and Trotter~\cite{DBLP:journals/dm/KiersteadT92} proved that the triangle-free Zykov graphs~\cite{Zykov1949} have the AOP property. The following proposition in~\cite{Girão2023} shows that any graph with girth sufficiently higher than its chromatic number has the AOP property.
\begin{proposition}[\cite{Girão2023}]\label{AOP property graphs}
    Let $n \geq 2$. Any graph $X_n$ with chromatic number
$n$ and girth greater than $2(n - 1)$ can be oriented acyclically such that there exists at most one directed path between any pair of vertices.
\end{proposition}

From Proposition~\ref{AOP property graphs}, we easily see that bipartite graphs have the AOP property and that planar graphs with girth strictly greater than $4$ also have the AOP property as their chromatic number is known to be bounded by $3$ (a classical theorem of Grötzsch~\cite{grotzsch1959dreifarbensatz}).

Observe that in order to have two disjoint directed paths between two vertices, one first needs to have a cycle. Now if the girth is large, one has intuitively more options to orient the edges of each cycle to make sure we do not have two vertices with two distinct directed paths between them. Hence a very natural question to ask here is whether all graphs with sufficiently large girth have the AOP property. We conjecture the following in relation to it.

\begin{conjecture} \label{conjecture} For any $g \in  \mathbb{N}$, there exist graphs with girth greater than or equal to $g$ that do not have the AOP property.
\end{conjecture}

It is easy to see that if a graph $H$ does not have the AOP property then any graph $G$ which has $H$ as a subgraph also does not have the AOP property. Hence, any graph $G$ with the AOP property has all the graphs that do not have the AOP property as forbidden subgraphs. Now let $G_g$ denote a graph with girth greater than or equal to $g$ that does not have the AOP property. Then by Proposition~\ref{AOP property graphs} we see that $\chi(G_g)$ tends to infinity as $g$ tends to infinity. Hence, Conjecture~\ref{conjecture} will trivially imply the existence of graphs with arbitrarily high chromatic number and girth that do not have the AOP property \footnote{We learned from Sophie Spirkl (personal communication) that she along with her co-authors managed to show that the conjecture indeed holds.}.
Although we are unable to resolve the conjecture, in this section we prove results of a similar nature and study the AOP property in the context of shift graphs and their induced subgraphs. We also prove the conjecture for $g=5$.

We first prove the following simple lemma that will be used to show that some graphs do not have the AOP property.

\begin{lemma}\label{orientation of n-cycle}
    Let $C= (u_1, u_2, u_3, \ldots, u_k)$ be a $k$-cycle. Let $\vec{C}$ be an orientation of $C$ that contains a directed path of length $k-2$. Then $\vec{C}$ is acyclic iff there exist two vertices in $\vec{C}$ such that there exist two disjoint directed paths between them.
\end{lemma}

\begin{proof}
    Clearly, if there exist two vertices in $\vec{C}$ such that there exist two disjoint directed paths between them, then $\vec{C}$ is acyclic. Without loss of generality let $\vec{P}= (u_1, u_2, u_3, \ldots, u_{k-1})$ be a directed path of length $k-2$ in $\vec{C}$.  Now assume $\vec{C}$ is acyclic. The following exhaustive set of cases may occur.
\begin{itemize}
     \item \emph{Case 1:} The vertex $u_k$ has an outgoing edge to both the vertices $u_1$ and $u_{k-1}$. 
     \item \emph{Case 2:} The vertex $u_k$ has an incoming edge from both the vertices $u_1$ and $u_{k-1}$. 
     \item \emph{Case 3:} The vertex $u_k$ has an outgoing edge to $u_{k-1}$ and an incoming edge from $u_1$. 
\end{itemize}
In the first case, there are two directed paths from $u_k$ to $u_{k-1}$. In the second case, there are two directed paths from $u_1$ to $u_{k}$. In the third case, there are two directed paths from $u_1$ to $u_{k-1}$. This finishes the proof.
\end{proof}

Observe that if we can prove that for any acyclic orientation of a fixed graph $G$, for some $k \in \mathbb{N}$ depending on the orientation, there is a $k$-cycle with a directed path of length $k-2$ in it, then the graph $G$ does not have AOP property by the above lemma. We will exploit this simple idea in showing the existence of graphs that do not have the AOP property.
We first show that there exist graphs of arbitrarily high odd-girth that do not have the AOP property. See figure~\ref{fig:odd} for a graph with odd-girth equal to $5$, which does not have the AOP property.

\begin{figure}
\begin{center}
\includegraphics{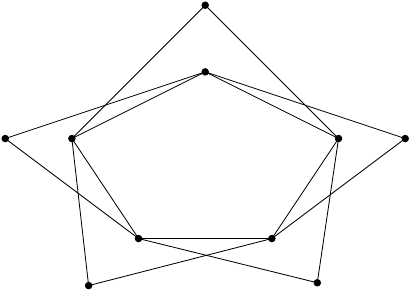}
\end{center}
\caption{A graph with odd-girth equal to $5$, that does not have the AOP property.}
\label{fig:odd}
\end{figure}

\begin{theorem} \label{high odd girth graphs not in AOP}
     For any $g\in \mathbb{N}$, there exists a graph $G$ with odd-girth greater than or equal to $g$, that does not have the AOP property.
\end{theorem}
\begin{proof}
    Let $g>4$ be a fixed odd natural number. Consider the graph $G$ that consists of a induced odd cycle $C=(u_1, u_2, u_3, \ldots, u_g)$ of length $g$ and a set $S=\{u'_1, u'_2, \ldots u'_g\}$ of $g$ vertices, such that $u'_i$ is only adjacent to the two neighbours of $u_i$ in the induced cycle $C$ (that is $u_{i-1}$ and $u_{i+1}$ for $i>1$ and $u_2$ and $u_g$ for $i=1$),  for all $1 \leq i \leq g$. Assume, $\vec{G}$ be an acyclic orientation of $G$. So the induced subdigraph $\vec{C}$ of $\vec{G}$ is also acyclic. Now since $\chi(C)=3$, by Proposition~\ref{coloring and orientation}, we know there exists a directed $2$-path $\vec{P}=(u_i, u_j, u_k)$ in $\vec{C}$. We know by construction that the vertex $u'_j \notin C$ is adjacent to the vertices $u_i$ and $u_k$ of $P$. Hence $\vec{C'}=(u'_j, u_i, u_j, u_k)$ is a directed $4$-cycle containing a directed path $\vec{P}$ of length $2$, so by Lemma~\ref{orientation of n-cycle}, there exist two vertices in $\vec{C'}$ with two disjoint directed paths between them. Hence we have proved that it is impossible to acyclically orient the graph $G$ such that between any two vertices there is at most one directed path.
    Next we will show that the odd-girth of $G$ is equal to $g$. Let $C^*$ be the smallest odd cycle of $G$, then $C^*$ has no chord otherwise we will get a smaller odd cycle than $C^*$. Now assume that $C^*$ has a length strictly less than $g$. Let $u'_i \in C^* \cap S$. We know that the neighbors of $u'_i \in C^*$ are vertices of $C$ that are also neighbors of the vertex $u_i$. Hence, if $u_i \in C^*$ then $C^*$ has a chord, contradicting minimality. So, if $u'_i \in C^*$, we have $u_i \notin C^*$ and so we can simply replace every vertex $u'_i \in C^* \cap S$ by the vertex $u_i \in C$ and still obtain a cycle of the same length. After all the replacements, the cycle consists only of vertices of $C$. As $C$ is an induced cycle of $G$, $C^*$ must have a length equal to $g$. This finishes the proof.
\end{proof}

We move on to prove Conjecture~\ref{conjecture} for the case when $g=5$.

\begin{theorem}
    There exists a graph $G$ whose girth is equal to $5$ that does not have the AOP property.
\end{theorem}
\begin{proof}
    Let $G_0$ be a graph with chromatic number $4$ and girth $5$. We know that such graphs exist~\cite{Descartes2,Descartes}. Now we will create a sequence of graphs iteratively as follows. Suppose $G_0, G_1, G_2, \ldots G_k$ is already defined. Now choose a $3$-path $P_k=(u_k, v_k, w_k) $ in $G_0$ which is not a part of any $5$-cycle in $G_k$, add a vertex $q_k$ and add two egdes connecting $q_k$ to the endpoints of the $3$-path namely ($u_k$ and $w_k$), thus creating a $5$-cycle. Let $G_{k+1}=G_k \cup \{q_k\}$ denote the new graph. It is to be noted that the $3$-path we are choosing is always a part of the initial graph $G_0$. This process must end after some iterations, since $G_0$ is a finite graph and has finitely many distinct paths of length $3$. Let $G=G_m$ be the final graph of this sequence. So every $3$-path of $G_0$ is a part of some $5$-cycle in $G$. Let $\vec{G}$ be any acyclic orientation of the graph $G$. Since $ \chi(G_0) \geq 4$, by Proposition~\ref{coloring and orientation} we know that there exists a directed $3$-path $\vec{P}$ of the induced subdigraph $\vec{G_0}$. By construction, $P$ belongs to some $5$-cycle $C$ of $G$. Hence by Lemma~\ref{orientation of n-cycle}, we know that there exists two vertices in $\vec{C}$ such that there exists two disjoint paths between them. Hence the graph $G$ does not have the AOP property. 
    
    Now we want to show that the girth of the graph is greater than or equal to $5$. Let $C=(a, b, c, d)$ be a $4$-cycle in $G$. Hence the four vertices of $C$ cannot all belong to $G_0$ as the girth of $G_0$ is equal to $5$. Since no two vertices in $G\setminus G_0$ are adjacent and any vertex in $G\setminus G_0$ is only adjacent to two vertices of $G_0$, we only have the following two cases.
\begin{itemize}

\item  \emph{Case 1: $a \in G\setminus G_0$ and $\{b, c, d\}\in G_0$.}

Observe that $c$ does not belong to any path of length $3$ starting at $b$ and ending at $d$ in $G_0$, otherwise, we will have a triangle in $G_0$. So any path of length $3$ starting at $b$ and ending at $d$ in $G_0$ is a part of a $5$-cycle in $G_0$ due to the presence of the vertex $c$. Hence the vertex $a$ clearly could never exist by construction. Hence we have a contradiction.

\item \emph{Case 2: $a, c \in G\setminus G_0$ and $\{b, d\}\in G_0$.}

Suppose $a \in G_i$ and $a \notin G_{i-1}$, similarly $c \in G_j$ and $c \notin G_{j-1}$. Without loss of generality assume $i<j$. Since $c \notin G_0$, $a$ does not belong to any $3$-path in $G_0$, starting at $b$ and ending at $d$. So any path of length $3$ starting at $b$ and ending at $d$ in $G_0$ is a part of a $5$-cycle in $G_i$ due to the addition of the vertex $a$. Hence the vertex $c$ clearly could never exist as it will violate the conditions of the construction of $G_j$. Hence we have a contradiction.
\end{itemize}   
Thus there cannot be any $4$-cycle in $G$. Now suppose $G$ has a triangle $T=(a, b, c)$. Clearly, all the vertices of $T$ cannot belong to $G_0$. Also since no two vertices in $G\setminus G_0$ are adjacent, only one vertex of the triangle $T$ must belong to $G\setminus G_0$. So without loss of generality assume $a \in G\setminus G_0$, and $b,c \in G_0$. By construction there must exist a $3$-path in $G_0$, starting at $b$ and ending at $c$. But since $b$ and $c$ are adjacent it will create a $4$-cycle in $G_0$ which is a contradiction. This finishes the proof. 
\end{proof}

Now we move on to study the AOP property in the context of shift graphs.
First, we show that $G_{9,2}$ does not have the AOP property. We need the following definition. Let $\vec{G}$ be a directed acyclic graph with $\sigma$ denoting a topological ordering of $\vec{G}$. Then for two adjacent vertices $u_1$ and $u_2$ in $L(\vec{G})$, we say that $u_1$ is a \emph{parent} of $u_2$ if $I(u_1)<I(u_2)$, where the index function $I$ is defined with respect to the topological ordering $\sigma$. Similarly, a vertex $u$ is said to be the parent of a bag $B_i$ if $u$ is adjacent to all the vertices of $B_i$ and $I(u)<i$.

\begin{theorem}
   The shift graph $G_{9,2}$ does not have the AOP property.
\end{theorem}
\begin{proof} 
We know from Observation~\ref{the shift graph} that $G_{9,2}=L(\vec{T}_9)$, where $\vec{T}_9$ is the directed acyclic complete graph on $9$ vertices. Let $\sigma=\{v_1, v_2, \ldots, v_9\}$ be the topological ordering of the vertices of the acyclic complete digraph $\vec{T_9}$. Let $\B_{\sigma}=\{B(i): i=1, 2, \ldots , 8\}$ to be the bag decomposition of $L(\vec{T}_9)$. Suppose for the sake of contradiction $L(\vec{T}_9)$ have the AOP property. So there exists an acyclic orientation $\overrightarrow{L}(\vec{T}_9)$ of $L(\vec{T}_9)$ such that between any two vertices of $\overrightarrow{L}(\vec{T}_9)$ there is at most one directed path. 

Let $B(i)$ be a bag of $L(\vec{T}_9)$ with at least two vertices $u, u^*$ and at least two parents $p,q$. First, observe that any two parents of $B(i)$ together with any two vertices of $B(i)$ form a $4$ cycle. Suppose the parent $p$ of $B(i)$ has an incoming edge from $u$ and an outgoing edge to $u^*$. So $(u, p, u^*)$ is a directed path of length $2$ in the 4-cycle formed by $p, q, u, u^*$. But this is a contradiction to Lemma~\ref{orientation of n-cycle}. So every parent of $B(i)$ either has outgoing edges to all the vertices of $B(i)$ or has incoming edges to all the vertices of $B(i)$.

Now if there is a vertex $u \in B(i)$ such that one parent $p$ of $B(i)$ has an incoming edge from $u$ and another parent $q$ of $B(i)$ has an outgoing edge to $u$ in $\overrightarrow{L}(\vec{T}_9)$. So the $4$-cycle created by vertices $p, q, u, u^*$ again has a 2-directed path which gives a contradiction to Lemma~\ref{orientation of n-cycle}. Hence, if a bag $B(i)$ has at least $2$ vertices and at least two parents then either all the parents of the bag $B(i)$ must have only outgoing edges to all the vertices of $B(i)$ or all the parents of $B(i)$ must have only incoming edges from all the vertices of $B(i)$. So for a bag $B(i)$ if all its parents have only outgoing edges to all its vertices then we call the bag a positive bag. Similarly if for a bag $B(i)$ if all its parents have only incoming edges from all its vertices, then we call the bag $B(i)$ a negative bag.

We know $|B(i)|  \geq 2$ for all $2 \leq i \leq 8$, and $B(i)$ has at least two parents for all $2 \leq i \leq 6$. Hence all the bags $B(i)$, $2 \leq i \leq 6$ must either be a positive bag or a negative bag. So by pigeon-hole principle, among them, there must exist $3$ positive bags or $3$ negative bags. Suppose $B(i), B(j), B(k)$ are positive bags with $2\leq i<j<k \leq 6$. Let $p_k$ be a parent of $B(k)$ and $u_i$ be a vertex of $B(i)$. Let $p_{i,j}$ be a parent of $B(i)$ in $B(j)$. Also let $p_{i,k}$, $p_{j,k}$ denote the parent of $B(i)$ and $B(j)$ in $B(k)$ respectively. Observe that $P_1= (p_k, p_{i,k}, u_i)$ is a directed path from $p_k$ to $u_i$. Also $P_2= (p_k, p_{j,k}, p_{i,j}, u_i)$ is also a directed path from $p_k$ to $u_i$. So we have $2$ directed paths from $p_k$ to $u_i$ which is a contradiction. Similarly, if there exists $3$ negative bags say $B(i), B(j), B(k)$, with $2\leq i<j<k \leq 6$, by the exact same argument we will have $2$ directed paths from a vertex of $B(i)$ to the parent of $B(k)$. So this creates a contradiction. This finishes the proof.
\end{proof}

It is natural to ask whether there are induced subgraphs of shift graphs with the AOP property and arbitrary high chromatic number. If we get an induced subgraph with girth sufficiently larger than the chromatic number then we would be fine by Proposition~\ref{AOP property graphs}. But girth is large implies the induced subgraph is $C_4$-free but we know by Theorem~\ref{Complete bipartite free induced}, that any $C_4=K_{2,2}$ free induced subgraph has chromatic number bounded. Hence it is not possible to prove the existence of induced subgraphs of shift graphs with the AOP property and arbitrary high chromatic number by directly using Proposition~\ref{AOP property graphs}. But interestingly there are such subgraphs as we show next. The key to our construction is a simple, yet fundamental property of line digraphs of acyclic graphs which we prove below.

\begin{lemma}\label{AOP property of S(G)}
    If a digraph $\vec{G}$ is acyclic and between any two of its vertices there exists at most one directed path then the line digraph $\vec{L}(\vec{G})$ is acyclic and between any two of its vertices there exists at most one directed path.
\end{lemma}
\begin{proof}
Note that by Observation~\ref{L(G) acyclic}, we know that $\vec{L}(\vec{G})$ is acyclic. Now let  $v_1,v_2, \ldots, v_n$ be a topological ordering $\sigma$ of the vertices of the directed acyclic graph $\vec{G}$. Let $\B_{\sigma}  =\{B(i): i=1, 2, \ldots , n-1\}$ be the bag decomposition of $L(\vec{G})$. Suppose for the sake of contradiction there are two distinct directed paths with ends $w$ and $w^*$ in $\vec{L}(\vec{G})$. It trivially follows that there must exist two vertices $u$ and $u^*$ in $\vec{L}(\vec{G})$, such that there are two disjoint directed paths between them. Let $\vec{P_1}=(u=u_0, u_1, u_2, \ldots, u_p, u_{p+1}=u^*)$ and $\vec{P_2}=(u=u'_0, u'_1, u'_2, \ldots, u'_q, u'_{q+1}=u^*)$ be two disjoint directed paths with ends $u$ and $u^*$. Also since $\vec{L}(\vec{G})$ is a simple digraph, w.l.o.g assume that $q \geq 1$ (that is, $u'_q \neq u, u^*$).

Recall that from Observation~\ref{structure of S(G)}(ii), we know that if $u_1, u_2\in L(\vec{G})$ are adjacent then the edge between $u_1$ and $u_2$ is directed from $u_1$ to $u_2$ in $L(\vec{G})$ iff $I(u_1)<I(u_2)$. So we have $I(u)<I(u_i)<I(u_j)<I(u^*)$ for all $1 \leq i<j \leq p$ and $I(u)<I(u'_i)<I(u'_j)<I(u^*)$ for all $1 \leq i<j \leq q$. Hence from Observation~\ref{structure of S(G)}(ii), we know that $\vec{W_1}=(v_{I(u)}, v_{I(u_1)}, v_{I(u_2)}, \ldots, v_{I(u_p)}, v_{I(u^*)})$ and $\vec{W_2}=(v_{I(u)}, v_{I(u'_1)}, v_{I(u'_2)}, \ldots, v_{I(u'_q)}, v_{I(u^*)})$ are two directed paths in the graph $\vec{G}$ from the vertex  $v_{I(u)}$ to the vertex $v_{I(u^*)}$. Now if $\vec{W_1}$ and $\vec{W_2}$ are distinct directed paths, we immediately have a contradiction by the property of the digraph $\vec{G}$ and the proof follows. 

Now since $u'_q \neq u, u^*$ and $\vec{P_1}$ and $\vec{P_2}$ are disjoint, we have that $u'_q \neq u_p$. So $u_p$ and $u'_q$ are distinct vertices of $L(\vec{G})$ which are both adjacent to $u^*$, so by Observation~\ref{structure of S(G)}(iv) we know that both $I(u_p) \neq I(u'_q)$. Hence $\vec{W_1} \neq \vec{W_2}$ and the proof follows.
\end{proof}
Let $Z_n$ denote the $n^{th}$ Zykov graph~\cite{Zykov1949}. Recall that it is known that the Zykov graphs can be acyclically oriented such that there exists at most one directed path between any two of its vertices~\cite{Davies23, Girão2023,DBLP:journals/dm/KiersteadT92}. Let $\vec{Z}_n$ denote an acyclic orientation of $Z_n$ such that there exists at most one directed path between any two of its vertices. It is also well known that the Zykov graph is triangle-free and $\chi({Z_n})$ goes to infinity as $n$ tends to infinity.
Now using the above lemma and the properties of the Zykov graph we can construct induced subgraphs of the shift graph $G_{n,2}$ with high chromatic number and odd-girth, that have the AOP property. We make this precise in the following theorem by showing that the underlying undirected graph $L^g(\vec{Z}_n)$ of the line digraph $\vec{L}^g(\vec{Z_n})$, which is also an induced subgraph of $G_{n,2}$ by Observation~\ref{the induced shift graph}, has all the required properties.

\begin{theorem}\label{AOP graphs with high chromatic number and odd girth} 
    Let $g \in \mathbb{N}$ be fixed. Then the graph $L^g(\vec{Z}_n)$, where $\vec{Z}_n$ is an acyclic orientation of the $n^{th}$ Zykov graph, can be acyclically oriented such that between any two vertices there exists at most one directed path. Moreover, the chromatic number of the graphs $L^g(\vec{Z}_n)$ goes to infinity as $n$ goes to infinity and $L^g(\vec{Z}_n)$ have odd-girth greater than or equal to $2g+3$ for all $n\in \mathbb{N}$.
\end{theorem}

\begin{proof}
    As $\chi({Z_n})$ goes to infinity as $n$ tends to infinity, using Corollary \ref{chromatic iteration}, we have that the chromatic number of the sequence of graphs $\{L^g(\vec{Z}_n)\}$ goes to infinity as $n$ goes to infinity. Now from Corollary \ref{girth iteration}, we know that for all $n \in \mathbb{N}$, the graph $L^g(\vec{Z}_n)$ have odd-girth greater than or equal to $2g+3$ as the graph $Z_n$ has odd-girth greater than or equal to $5$. And finally applying Lemma \ref{AOP property of S(G)}, $g$ times, we know that for all $n \in \mathbb{N}$, $\vec{L}^g(\vec{Z_n})$ is acyclic and between any two of its vertices there exists at most one directed path. This finishes the proof.
\end{proof}

It is to be noted that the existence of graphs with large girth and chromatic number with the AOP property were already known~\cite{Girão2023} and can also be constructed. But our construction, although weaker in a certain sense (only large odd-girth), is very different. Specifically, our construction is an induced subgraph of the shift graph and the orientation that proves AOP property is none other than the induced natural orientation of the shift graph. As a consequence, this also shows that there exists graphs with arbitrarily large chromatic number, whose AOP orientation does not admit a $4$ vertex path $\vec{P}_4$ oriented $\rightarrow\leftarrow\rightarrow$ as an induced subdigraph.

\section{Concluding remarks}
\label{sec:conclusions}
One important open question that still remains is whether the constant $c_F$ in Theorem~\ref{girao} is bounded by a function of $\chi(F)$. More specifically for the class of complete bipartite graphs, is $c_F$ uniformly bounded by a constant or is it still a function of the number of vertices of $F$? It would also be nice to find connections between the AOP property of graphs with other graph parameters or structural properties.

\section{Acknowledgements}
I would like to thank the anonymous reviewer for detailed comments that helped to improve the overall quality of the paper and also for pointing out that the bound previously obtained in Theorem~\ref{Complete bipartite free induced} can be improved. I would like to thank Sophie Spirkl for some thoughtful discussions. I would also like to thank Mark de Berg for reading the paper carefully and helping to improve the writing of the paper. Finally, I thank Frits Spieksma, Shivesh Roy and Rounak Ray for some useful suggestions. The work is supported by the  Dutch Research Council (NWO) through Gravitation-grant NETWORKS-024.002.003.

\bibliography{references}

\newpage

\end{document}